\documentclass{amsart}
\usepackage{amsthm}
\usepackage{mathrsfs} 
\usepackage{amsmath}
\usepackage{graphicx}
\usepackage{amsfonts}
\usepackage{amssymb}
\usepackage{inputenc}
\usepackage[bookmarks, colorlinks=false, plainpages = false, citecolor = green, urlcolor = blue, filecolor = blue, linkcolor = blue]{hyperref}

\newcommand{\h}{\mathfrak{h}}
\newcommand{\g}{\mathfrak{g}}
\newcommand{\p}{\mathfrak{p}}
\newcommand{\m}{\mathfrak{m}}
\newcommand{\bs}{\backslash}

\newcommand{\bd}{\begin{displaymath}}
\newcommand{\ed}{\end{displaymath}}

\newtheorem{lema}{Lemma}[section]
\newtheorem{prop}[lema]{Proposition}
\newtheorem{cor}[lema]{Corollary}
\newtheorem{teo}[lema]{Theorem}

\theoremstyle{remark}

\newtheorem{remark}{\textbf{Remark}}[section]

\begin{document}

\title{Isometry groups of generalized Stiefel manifolds}
\author{Manuel Sedano-Mendoza}
\email{manuel.sedano@cimat.mx}


\maketitle

\begin{abstract}
A generalized Stiefel manifold is the manifold of orthonormal frames in a vector space with a non-degenerated bilinear or hermitian form. In this article, the Isometry group of the generalized Stiefel manifolds are computed at least up to connected components in an explicit form. This is done by considering a natural non-associative algebra $\m$ associated to the affine structure of the Stiefel manifold, computing explicitly the automorphism group $Aut(\m)$ and inducing this computation to the Isometry group.
\end{abstract}

\section{Introduction}

Homogeneous spaces are a very important family of manifolds where there is a Lie group acting transitively and so, the manifold can be written as a quotient $G \backslash H$, where $G$ is a closed subgroup of the Lie group $H$. When the homogeneous manifold comes with a geometric structure (such as Riemannian, pseudo-Riemannian or affine) preserved by the action of $H$, then an elementary and very important problem is to determine the whole isometry group, at least up to connected components. In the case where $H$ is a compact Lie group, Onishchik gave a general classification of the isometry group \cite{Oni2} that led to the computation of isometry groups of special homogeneous spaces such as isotropy irreducible, normal homogeneous and naturally reductive spaces, see \cite{Reg}, \cite{w-z}. If on the other hand $K \backslash H$ is a Riemannian symmetric space with $H$ semisimple and acting faithfully on $K \backslash H$, then a classical result \cite{He} tells us that $H$ realizes the connected component of $Iso(K \backslash H)$, in general however it is still an open problem to compute the isometry group for $G \bs H$ when $H$ is non-compact even for the case of pseudo-Riemannian symetric spaces.

A big family of pseudo-Riemannian homogeneous manifolds that are neither symmetric nor compact are the generalized Stiefel manifolds. Let $H$ be the linear group preserving a non-degenerated bilinear or hermitian form, and let $G \times L$ be a Lie subgroup of block-diagonal matrices of fixed dimensions so that $L \backslash H$ represents a generalized Stiefel manifold of orthogonal frames of fixed signature and $(G \times L) \backslash H$ represents the Grassmann manifold of non-degenerated planes of fixed dimension, the pairs $(H, G \times L)$ are listed in Table \ref{tabla} and we call them symmetric pairs of Grassmann type, see section \ref{sec:stief} for a more precise description.

\begin{table}[htbp]
\begin{center}\begin{tabular}{|c|c|c|c|}
	\hline
	$H$ & $ G \times L$ & $L \bs H$ & Type \\
	\hline
	$SO(p,q)$ & $ O(p-k,q-l) \times SO(k,l)$ & $SO(k,l) \bs SO(p,q)$ & Orthogonal \\
	\hline
	$SU(p,q)$ & $ U(p-k,q-l) \times SU(k,l)$ & $SU(k,l) \bs SU(p,q)$ & Hermitian \\
	\hline
	$Sp(p,q)$ & $ Sp(p-k,q-l) \times Sp(k,l)$ & $Sp(k,l) \bs Sp(p,q)$ & Hermitian \\
	\hline
	$Sp(2n,\mathbb{R})$ & $Sp(2(n-m),\mathbb{R}) \times Sp(2m,\mathbb{R})$ & $Sp(2m,\mathbb{R}) \backslash Sp(2n,\mathbb{R})$ & Symplectic \\
	\hline
	$Sp(2n,\mathbb{C})$ & $Sp(2(n-m),\mathbb{C}) \times Sp(2m,\mathbb{C})$ & $Sp(2m,\mathbb{C}) \backslash Sp(2n,\mathbb{C})$ & Symplectic \\
	\hline
	$SO(n,\mathbb{C})$ & $SO(n-m,\mathbb{C}) \times SO(m,\mathbb{C})$ & $SO(m,\mathbb{C}) \backslash SO(n,\mathbb{C})$ & Orthogonal \\
	\hline
\end{tabular}
	\caption{Symmetric pairs of Grassmann type.}
	\label{tabla}
\end{center}
\end{table}

The pseudo-Riemannian structure in $H$ induced by the Killing form is invariant under left and right multiplications by $H$, so it induces a pseudo-Riemannian structure on the Stiefel manifold $L \bs H$ that has left and right multiplication isometries, more precisely, there is a homomorphism
	\[	M : G \times H \rightarrow Iso(L \backslash H),	\]
given by multiplication elements
	\[	M(g,h) = L_g \circ R_{h^{-1}}, \qquad L_g \circ R_{h^{-1}} (L h') = L (gh'h^{-1}),	\]
for all $g \in G$ and $h \in H$. The purpose of the present paper is to prove the following

\begin{teo}\label{Thm:main}
Let $(H,G \times L)$ be a symmetric pair of Grassmann type, so that $L \backslash H$ is a connected generalized Stiefel manifold and $G$ has non-trivial simple part, then the isometry group $Iso(L \bs H )$ has finitely many connected components and has the group
	\[	M(G \times H) \cong (G \times H) / \Gamma, \]
as a finite index subgroup, where $\Gamma \subset G \times H$ is a finite central group. In particular, the Lie algebra of $Iso(L \bs H)$ is the same as the Lie algebra of $G \times H$.
\end{teo}

\begin{remark}\label{rem:non-connected}
	Given a symmetric pair of Grassmann type $(H, G \times L)$, $H$ is connected unless it is isomorphic to $SO(p,q)$. In any case, $L \backslash H$ has at most four connected components, all isometric to $S \backslash H_0$ for some fixed subgroup $S$ of the connected component of the identity $H_0$ of $H$. Thus if $L \backslash H$ has $k$-connected components, then $Iso(L \backslash H)$ has as a finite index subgroup $I_1 \times ... \times I_k$, where $I_j \cong Iso(S \backslash H_0)$ and $I_j$ is as Theorem \ref{Thm:main}.
\end{remark}

We may observe that the previous result includes the case where $L$ is the trivial group that corresponds to the Stiefel manifold of complete ``positively oriented" orthonormal frames, so we recover the classical result where $Iso(H) = M(H \times H)$ for the bi-invariant pseudo-Riemannian structure in $H$ induced by the Killing form, in fact the proof of Theorem \ref{Thm:main} presented here follows the strategy in \cite{Q1} to prove this fact and it is an extension of a special case computed in \cite{Man}. 

Theorem \ref{Thm:main} was originally motivated by the problem of finding Compact Clifford-Klein forms on the homogeneous spaces $L \bs H$, that is, the problem of finding a discrete group $\Gamma$ acting properly on $L \bs H$ such that $(L \bs H) / \Gamma$ is compact. In  \cite{Bor}, A. Borel proved that there always exists a discrete subgroup $\Gamma \subset H$ so that $L \bs H / \Gamma$ is a compact Clifford-Klein form when $L \subset H$ is a compact subgroup, and in \cite{Z5} R. Zimmer proved that there doesn't exist such compact Clifford-Klein forms for quotients of the form $SL_m(\mathbb{R}) \bs SL_n(\mathbb{R})$, for $m < n/2$, because in such cases we have the simple Lie group $SL_{n-m}(\mathbb{R})$ acting as isometries. In fact, Zimmer conjectured that if the homogeneous space $L \bs H$ admits a $G$-action via left multiplications with $G$ simple non-compact and $L$ non-compact, then $L \bs H$ doesn't admit compact Clifford-Klein forms, such statement is a reformulation of one of a series of conjectures on $G$-actions known as Zimmer program \cite{Z2}. A direct consequence of Theorem \ref{Thm:main} is that, in order to find compact geometric quotients in generalized Stiefel manifolds it is enough to look at the possible co-compact actions of discrete subgroups of $H$, more precisely we have

\begin{cor}\label{Thm:Cliff-Klein}
Let $(H , G \times L)$ be a symmetric pair of Grassmann type such that $L \bs H$ is a generalized Stiefel manifold and $G$ has non-trivial simple part. If there exist a $G$-equivariant pseudo-Riemannian covering $L \bs H \rightarrow \left( L \bs H \right) / \Gamma$, then there is a discrete subgroup $\Gamma_0 \subset H$ such that
	\begin{itemize}
		\item if $\left( L \bs H \right) / \Gamma$ is compact (has finite volume), then $L \bs H /\Gamma_0$ is compact (has finite volume),
		
		\item if $\left( L \bs H \right) / \Gamma$ has finite volume and $L$ compact, then $\Gamma_0$ is a Lattice in $H$.
	\end{itemize}
\end{cor}

The relationship between the discrete subgroups $\Gamma_0$ and $\Gamma$ is given by successive extensions of finite groups, more precisely, there is a finite index subgroup $\Gamma' \subset \Gamma$, a finite subgroup $F \subset Z(G)$ and an exact sequence of groups
	\[	1 \rightarrow F \rightarrow \Gamma' \rightarrow \Gamma_0 \rightarrow 1,	\]
we say that $\Gamma$ and $\Gamma_0$ are commesurable if they are related in such a way.
In the particular case where $Z(G)$ is finite, $\Gamma_0$ can be obtained as a finite index subgroup of $\Gamma$ and we have a finite covering
	\[	L \bs H /\Gamma_0 \rightarrow \left( L \bs H \right) / \Gamma.	\]

In \cite[Thm. 29]{Co}, extending Zimmer's result, D. constantine obtained a non-existence result of compact Clifford-Klein forms in a larger family of homogeneous spaces that include some of the generalized Stiefel manifolds but with the restriction that the group $\Gamma$ acting co-compactly is in fact a discrete subgroup of $H$. Thus Corollary \ref{Thm:Cliff-Klein} combined with Constantine's result gives us the following general result on compact geometric quotients

\begin{cor}
Let $(H , G \times L)$ be a symmetric pair of Grassmann type such that $L \bs H$ is a generalized Stiefel manifold, where $G$ is non-compact reductive group, with real rank at least 2. If there exist a $G$-equivariant pseudo-Riemannian covering $L \bs H \rightarrow \left( L \bs H \right) / \Gamma$, with $\left( L \bs H \right) / \Gamma$ compact, then $L$ is compact and $\Gamma$ is commesurable to a lattice $\Gamma_0 \subset H$.
\end{cor}

\begin{remark}
	As ilustrated by Corollary \ref{Thm:Cliff-Klein}, a discrete group $\Gamma$ acting isometrically and co-compactly in a generalized Stiefel manifold $L \backslash H$ can be obtained via some extensions of a co-compact lattice $\Gamma_0 \subset H$ by finite groups, so it may be interesting to characterize isometric actions of finite groups on generalized Stiefel manifolds. A first approach to this is given by the fact that $Iso(L \backslash H)$ is a Lie group with only finitely many components (as stated in Theorem \ref{Thm:main}) and according to \cite[Thm. 2]{Pop}, this group is Jordan. Being a Jordan group tells us that the list of all of its finite subgroups are extensions of abelian groups by groups from a finite list. A consequence of this is that, up to isomorphism, only finitely many finite simple groups can act isometrically in $L \bs H$.
\end{remark}

\section{Generalized Stiefel and Grassmann manifolds}\label{sec:stief}

Let $\mathbb{F}$ be the real, complex or quaternionic numbers denoted by $\mathbb{R}$, $\mathbb{C}$ and $\mathbb{H}$ correspondingly and denote by $G_n(\mathbb{F}^N)$ the Grassmann manifold of $\mathbb{F}$-subspaces of $\mathbb{F}^N$ of dimension $n$. In this section we describe submanifolds of $G_n(\mathbb{F}^N)$ associated to the different non-degenerated bilinear and hermitian forms over $\mathbb{F}^N$ and the corresponding manifolds of orthonormal frames called generalized Grassmann and Stiefel manifolds correspondingly. In this section we denote by $I_k$ the $k \times k$-identity matrix, for every $k \in \mathbb{N}$.

\subsection{Grassmann symmetric pairs of Symplectic type}
For every even integer $k$, consider the symplectic structure in $\mathbb{F}^k$ given by
	\[	\omega_k(X,Y) = Y^t J_k X,	\qquad 	J_k = \left(\begin{array}{cc} 0 & I_{k/2} \\ -I_{k/2} & 0 \end{array}\right),	\]
with symmetry group $Sp(k,\mathbb{F}) = \{ X \in M_{k \times k}(\mathbb{F}) : X J_k X^t = J_k  \}$. Given $n,m \in \mathbb{Z}$ even numbers, we observe that the symplectic structures in $\mathbb{F}^{n+m}$ given by $J_{n+m}$ and $\left(\begin{array}{cc} J_n & 0 \\ 0 & J_m \end{array}\right)$ are equivalent, so we have a decomposition of Lie algebras
	\[	\mathfrak{sp}(n+m,\mathbb{F}) = \mathfrak{sp}(n,\mathbb{F}) \oplus \mathfrak{sp}(m,\mathbb{F}) \oplus M_{n \times m}(\mathbb{F}), \]
where we consider the block-diagonal inclusion
	\[	\mathfrak{sp}(n,\mathbb{F}) \oplus \mathfrak{sp}(m,\mathbb{F}) \rightarrow \mathfrak{sp}(n+m,\mathbb{F}), \qquad
		(X, Y) \mapsto \left(\begin{array}{cc} X & 0 \\ 0 & Y \end{array}\right)	\]
and
	\[	M_{n \times m}(\mathbb{F}) \rightarrow \mathfrak{sp}(n+m, \mathbb{F}), \qquad A \mapsto \left(\begin{array}{cc} 0 & A \\ -\widehat{A} & 0 \end{array}\right),	\]
for $\widehat{A} = - J_m A^t J_n$, observe that
	\[	\omega_n (A x,y) = \omega_m(x, \widehat{A} y),	\qquad \forall \ x \in \mathbb{F}^m, \ y \in \mathbb{F}^n.	\]

An $\mathbb{F}$-subspace $V \subset \mathbb{F}^k$ is non-degenerated if $\omega_k$ is non-degenerated when restricted to $V$, $Sp(n+m, \mathbb{F})$ acts transitively on the set of non-degenerated subspaces of fixed dimension, thus on the generalized Grassmann manifold
	\[	X_n = \{ V \in G_n(\mathbb{F}^{n+m}) : V \textrm{\ non-degenerated } 	\},	\]
and we have the identification
	\[	X_n  \cong Sp(n, \mathbb{F}) \times Sp(m, \mathbb{F}) \backslash Sp(n+m, \mathbb{F}).	\]
An ordered set $\{e_1,...,e_k, e^1,...,e^k\}$ is a $\omega_{n+m}$-orthonormal $k$-frame if 
	\[	\omega_{n+m}(e_i,e_j) = \omega_{n+m}(e^i,e^j) = 0, \qquad	\omega_{n+m}(e_i,e^j) = \delta_{ij},	\]
in particular it spans a $2k$-dimensional non-degenerated subspace. Again $Sp(n+m, \mathbb{F})$ acts transitively on the set of $\omega_{n+m}$-orthogonal $k$-frames and so
	\[	Y_n = \{ F \subset \mathbb{F}^{n+m} : F \textrm{ is } \omega_{n+m}\textrm{-orthonormal }  n\textrm{-frame} \} \cong Sp(m, \mathbb{F}) \backslash Sp(n+m, \mathbb{F}).	\]
$Y_n$ is a generalized Stiefel manifold of symplectic type and 
	\[	(Sp(n+m, \mathbb{F}), Sp(n, \mathbb{F}) \times Sp(m, \mathbb{F}))	\] 
is a symmetric pair of Grassmann type.

\subsection{Grassmann symmetric pairs of Unitary type}
For every $n = (n_1,n_2) \in \mathbb{N}^2$ denote $|n|= n_1 + n_2$ and consider the hermitian space $\mathbb{F}^n$, that is $\mathbb{F}^{|n|}$ with the hermitian structure 
	\[	\langle x,y \rangle_n = y^* J_n x,	\qquad 	J_{(n_1,n_2)} = \left(\begin{array}{cc} I_{n_1} & 0 \\ 0 & -I_{n_2} \end{array}\right),	\]
$(n_1,n_2)$ is called the signature of the hermitian structure. The symmetry groups of the hermitian structure are given by the unitary group
	\[	U(\mathbb{F}^n) = \{X \in M_{|n| \times |n|}(\mathbb{F}): X J_n X^* =  J_n \}	\]
and the special unitary group $SU(\mathbb{F}^n) = \{X \in U(\mathbb{F}^n): det(X) = 1 \}$. The coincidences with the standard notation in the literature are
	\[	O(n_1,n_2) = U(\mathbb{R}^n), 	\quad	
			U(n_1,n_2) = U(\mathbb{C}^n), 
	\quad Sp(n_1,n_2) = U(\mathbb{H}^n), \]
thus at the Lie algebra level, $\mathfrak{u}(\mathbb{F}^n)$ is strictly bigger than $\mathfrak{su}(\mathbb{F}^n)$ only in the complex case. The hermitian structures in $\mathbb{F}^{|n|+|m|}$ given by $J_{n+m}$ and $\left(\begin{array}{cc} J_n & 0 \\ 0 & J_m \end{array}\right)$ are equivalent, so we have a decomposition
	\[	\mathfrak{su}(\mathbb{F}^{n+m}) = \mathfrak{s} \left( \mathfrak{u}(\mathbb{F}^n) \oplus \mathfrak{u}(\mathbb{F}^m) \right) \oplus M_{|n| \times |m|}(\mathbb{F}), \]
with the block-diagonal subgroup 
	\[	\mathfrak{s} \left( \mathfrak{u}(\mathbb{F}^n) \oplus \mathfrak{u}(\mathbb{F}^m) \right) = 
			\left\{ \left(\begin{array}{cc} X & 0 \\ 0 & Y \end{array}\right) : X \in \mathfrak{u}(\mathbb{F}^n), Y \in \mathfrak{u}(\mathbb{F}^m), \quad tr(X + Y) = 0 \right\}.	\]
and the inclusion
	\[	M_{|n| \times |m|}(\mathbb{F}) \rightarrow \mathfrak{su}(\mathbb{F}^{n+m}), \qquad A \mapsto \left(\begin{array}{cc} 0 & A \\ -\widehat{A} & 0 \end{array}\right),	\]
for $\widehat{A} = J_m A^* J_n$, as in the symplectic case we have
	\[	\langle Ax,y \rangle_n = \langle x, \widehat{A}y \rangle_m,	\qquad \forall \ x \in \mathbb{F}^m, \ y \in \mathbb{F}^n.	\]

An $\mathbb{F}$-subspace $V \subset \mathbb{F}^k$ is non-degenerated if $\langle \cdot , \cdot \rangle_k$ is non-degenerated when restricted to $V$, in that case, $\langle \cdot , \cdot \rangle_k|_{V \times V}$ will be an hermitian structure in $V$ and thus will have a defefinite signature. $SU(\mathbb{F}^{n+m})$ acts transitively on the set of non-degenerated subspaces of fixed dimension and fixed signature, so we have generalized Grassmann manifold
	\[	X_n = \{ V \in G_{|n|}(\mathbb{F}^{n+m}) : V \textrm{\ non-degenerated of signature } n 	\},	\]
and we have the identification
	\[	X_n \cong S \left( U(\mathbb{F}^{n}) \times U(\mathbb{F}^{m}) \right) \backslash SU(\mathbb{F}^{n+m}) .	\]
An ordered set $\{e_1,...,e_{|k|}\}$ is an orthonormal $k$-frame if $\langle e_i,e_j \rangle_{n+m}= \pm \delta_{ij}$ and spans a $|k|$-dimensional non-degenerated subspace of signature $k$. Again $SU(n+m, \mathbb{F})$ acts transitively on the set of orthogonal $k$-frames if $|k| < |n+m|$ and so
	\[	Y_n = \{ F \subset \mathbb{F}^{n+m} : F \textrm{ is } \textrm{orthonormal }  n\textrm{-frame} \} \cong SU(\mathbb{F}^{m}) \backslash SU(\mathbb{F}^{n+m}),	\]
for $|n| > 0$, or 
	\[	Y_{n+m} = \{ F \subset \mathbb{F}^{n+m} : F \textrm{ is } \textrm{orthonormal }  \textrm{basis} \} \cong U(\mathbb{F}^{n+m}).	\]
$Y_n$ is a generalized Stiefel manifold of unitary type and
	\[	(SU(\mathbb{F}^{n+m}), S \left( U(\mathbb{F}^{n}) \times U(\mathbb{F}^{m}) \right))	\]
is a symmetric pair of Grassmann type, our convention is $G = U(\mathbb{F}^{n})$ and $L = SU(\mathbb{F}^{m})$.

\subsection{Grassmann symmetric pairs of orthogonal type}
If instead of the hermitian inner product, we take a symmetric bilinear form in $\mathbb{C}^n$, then an analogous discusion follows and we get the decomposition 
	\[	\mathfrak{so}(n+m, \mathbb{C}) = \mathfrak{so}(n, \mathbb{C}) \oplus \mathfrak{so}(m, \mathbb{C}) \oplus M_{n \times m}(\mathbb{C})	\]
with the block-diagonal subgroup 
	\[	\mathfrak{so}(n, \mathbb{C}) \oplus \mathfrak{so}(m, \mathbb{C}) = 
			\left\{ \left(\begin{array}{cc} X & 0 \\ 0 & Y \end{array}\right) : X \in \mathfrak{so}(n, \mathbb{C}), \ Y \in \mathfrak{so}(m, \mathbb{C})	\right\}.	\]
and the inclusion
	\[	M_{n \times m}(\mathbb{C}) \rightarrow \mathfrak{so}(\mathbb{C}^{n+m}), \qquad A \mapsto \left(\begin{array}{cc} 0 & A \\ -A^t & 0 \end{array}\right).	\]
In the same way, we have the Grassmann manifolds
	\[	X_n = \{ V \in G_{n}(\mathbb{C}^{n+m}) : V \textrm{\ non-degenerated} 	\},	\]
with the identification
	\[	X_n \cong S \left( O(n, \mathbb{C}) \times O(m, \mathbb{C}) \right) \backslash SO(n+m, \mathbb{C}),	\]
the corresponding Stiefel manifolds
	\[	Y_n = \{ F \subset \mathbb{C}^{n+m} : F \textrm{ is } \textrm{orthonormal }  n\textrm{-frame} \} \cong SO(m,\mathbb{C}) \backslash SO(n+m,\mathbb{C}),	\]
and symmetric pair of Grassmann type 
	\[	(SO(n+m, \mathbb{C}), S \left( O(n, \mathbb{C}) \times O(m, \mathbb{C}) \right) ).	\]

\subsection{Recovering structure from brackets}
Let $(H , G \times L)$ be a symmetric pair of Grassmann type defined by a bilinear or a hermitian form $B_{n+m}$ in $\mathbb{F}^{n+m}$, we have a Lie algebra decomposition
	\[	\mathfrak{h} = \g \oplus \mathfrak{l} \oplus \mathfrak{p},	\]
with the Lie bracket meeting the following properties
	\[	[\g, \mathfrak{l}] = 0, \qquad [\g \oplus \mathfrak{l}, \mathfrak{p}] \subset \mathfrak{p}.	\]
As we saw, we have an identification 
	\[	\mathfrak{p} = \left\{	\left(\begin{array}{cc} 0 & A \\ -\widehat{A} & 0 \end{array}\right): A \in M_{n \times m}(\mathbb{F}) 	\right\} \cong	M_{n \times m}(\mathbb{F}), 	\]
		where $B_n(Ax,y) = B_m(x, \widehat{A}y)$, with the corresponding identifications $\g \subset M_{n \times n}(\mathbb{F})$ and $\mathfrak{l} \subset M_{m \times m}(\mathbb{F})$, the Lie bracket action becomes into the corresponding right and left multiplication of matrices
			\[	(\g \oplus \mathfrak{l} ) \times \mathfrak{p} \rightarrow \mathfrak{p}, \qquad (X + Y, A) \mapsto XA - AY.	\]
Moreover, the Lie bracket in $\h$ restricted to $\mathfrak{p}$ and then projected to the $\mathfrak{g}$-factor in $\h$ defines a bilinear form
	\[	\mathcal{B}: M_{n \times m}(\mathbb{F}) \times M_{n \times m}(\mathbb{F}) \rightarrow \mathfrak{g} \subset M_{n \times n}(\mathbb{F})	\]
given by $\mathcal{B}(A,B) = B \widehat{A} -A \widehat{B}$.

The previous facts tells us that for every $X \in \mathfrak{l}$, the $\mathbb{R}$-linear map
	\[	R_X : M_{n \times m}(\mathbb{F}) \rightarrow M_{n \times m}(\mathbb{F}), \qquad R_X(A) = AX,	\]
commutes with the $\g$-action in $M_{n \times m}(\mathbb{F})$ and $D_X \mathcal{B} = 0$, where 
	\[	D_{R_X} \mathcal{B}(A,B) = \mathcal{B} (R_X(A),B) + \mathcal{B} (A,R_X(B)), \qquad \forall \ A,B \in M_{n \times m}(\mathbb{F})  \]
is the derived bilinear form, this last fact is a consequence of the Jacobi identity. The following proposition tells us that in fact this is a complete characterization of the Lie subalgebra $\mathfrak{l}$.

\begin{prop}\label{prop:2-1}
If $Hom_\g(M_{n \times m}(\mathbb{F}))$ denotes the space of $\g$-homomorphisms in $M_{n \times m}(\mathbb{F})$, i.e. $\mathbb{R}$-linear maps commuting with the $\g$-action, then
	\[	\mathfrak{l}_0 = \{ F \in Hom_\g(M_{n \times m}(\mathbb{F})) : D_F \mathcal{B} = 0 \},	\]
	where $\mathfrak{l}_0 = \mathfrak{l} \oplus \mathfrak{u}(1)$ if $\h = \mathfrak{su}(\mathbb{F}^{n+m})$ and $\mathfrak{l} = \mathfrak{l}_0$ otherwise.
\end{prop}

\begin{lema}\label{lema:2-1}
If $F \in Hom_\g(M_{n \times m}(\mathbb{F}))$, then there exist $X \in M_{m \times m}(\mathbb{F})$ such that $F(B) = B X$.
\end{lema}

\begin{proof}
Observe that $\mathbb{F}^n$ considered as $n \times 1$ matrices has the structure of a $GL_n(\mathbb{F}) \times \mathbb{F}$-module given by
	\[ GL_n(\mathbb{F}) \times \mathbb{F} \times \mathbb{F}^n \rightarrow \mathbb{F}^n, \qquad
		(g,\lambda) \cdot x = g x \lambda,	\]
and if we consider the identification
	\[	\mathbb{F}^n \cong (\mathbb{F}^n)^*, \qquad \zeta \mapsto \zeta^*, \]
where $\zeta^*$ denotes the conjugated transposed element, then we have the isomorphism $\mathbb{F}^n \otimes (\mathbb{F}^m)^* \cong M_{n \times m}(\mathbb{F})$ together with the left action
	\[ GL_n(\mathbb{F}) \times GL_m(\mathbb{F}) \times M_{n \times m}(\mathbb{F}) \rightarrow M_{n \times m}(\mathbb{F}), \qquad
		(g,h) \cdot A = g A h^*,	\]
moreover, the tensor product $\mathbb{F}^n \otimes (\mathbb{F}^m)^*$ has the $\mathbb{F}$-equivariant property 
	\[ (u \lambda) \otimes v = u \otimes (\lambda v), \qquad \forall \ u \in \mathbb{F}^n, \ v \in (\mathbb{F}^m)^*, \ \lambda \in \mathbb{F}. \]
Take $\{e_i\}$ a basis of $(\mathbb{F}^m)^*$ over $\mathbb{F}$, so that we may write every element $\omega \in \mathbb{F}^n \otimes (\mathbb{F}^m)^*$ as
	\[ \omega = \sum_l \omega_l \otimes e_l, \qquad \omega_l \in \mathbb{F}^n, \]
now if we write $F(\zeta \otimes e_l)  = \sum_s F_s(\zeta, e_l) \otimes e_s$, as $F$ commutes with the $\g$-action, we can see that for every $s$,$l$,
	\[ F_{s,l} : \mathbb{F}^n \rightarrow \mathbb{F}^n, \qquad F_{s,l}(\zeta) = F_s(\zeta, e_l) \]
	is a $\g$-homomorphism so that there exist a scalar $A_{s,l} \in \mathbb{F}$, with $F_s(\zeta) = \zeta A_{s,l}$, this follows from the fact that $\g$ acts irreducibly in $\mathbb{F}^n$, and then we have an identification $Hom_\g(\mathbb{F}^n) \cong \mathbb{F}$, i.e. every $\g$-homomorphism on $\mathbb{F}^n$ can be realized by the right multiplication of a $\mathbb{F}$-scalar \cite{Oni}. We have thus
	\[	F(\zeta \otimes e_l) = \sum_s \zeta \otimes (A_{s,l} e_s) = \zeta \otimes (A e_l), \]
where $A$ is the matrix with coefficients $A_{s,l}$, if we define $X = A^*$, the result follows.
\end{proof}

\begin{lema}\label{lema:2-2}
For every $X \in M_{m \times m}(\mathbb{F})$, $D_{R_X} \mathcal{B} = 0$ if and only if 
		\[	B_m(Xu,v) + B_m(u,Xv) = 0, \qquad \forall \ u,v \in \mathbb{F}^m.  \]
\end{lema}

\begin{proof}
Denote by $X^\dag$, either the transpose or the conjugated transpose matrix of $X$, so that the bilinear form in $\mathbb{F}^N$ defining the Lie algebra $\h$ is given by
	\[	B_N(x,y) = x^\dag K_N y,	\]
and $K_N$ is a square matrix such that $K_N^2 = \varepsilon I_N$, $\varepsilon = \pm 1$. With this notation, $\widehat{A} = \varepsilon K_m A^\dag K_n$ and thus
	\[	D_X \mathcal{B} (A,B) = \varepsilon B \widetilde{X} A^\dag K_n - \varepsilon A \widetilde{X} B^\dag K_n,	\]
where $\widetilde{X} = X K_m + K_m X^\dag$, thus $D_X \mathcal{B} = 0$ if and only if $\widetilde{X} = 0$ and the result follows.
\end{proof}

\begin{proof}[Proof of Proposition \ref{prop:2-1}]
Take $F \in Hom_\g(M_{n \times m}(\mathbb{F}))$, by Lemma \ref{lema:2-1}, there exists an element $X \in M_{m \times m}(\mathbb{F})$ such that $F(B) = B X$. If moreover $D_F \mathcal{B} = 0$, Lemma \ref{lema:2-2} tells us that 
	\begin{equation}\label{skew}
		B_m(Xu,v) + B_m(u,Xv) = 0, \qquad \forall \ u,v \in \mathbb{F}^m,	
	\end{equation}
We have that $X \in M_{n \times m}(\mathbb{F})$ belongs to $\mathfrak{l}$ if and only if (\ref{skew}) holds and $tr(X) = 0$, and the only case where property (\ref{skew}) doesn't imply $tr(X) = 0$ is in the case $\h = \mathfrak{su}(\mathbb{F}^{n+m})$.
\end{proof}

\section{Affine structure of the homogeneous space}

Recall from \cite{He}, that the Killing form of a semisimple Lie group $H \subset GL_N(\mathbb{F})$ is given by a rescaling of the bilinear form $B(XY) = Re(\textrm{tr}(XY))$. Consider $(H, G \times L)$ a symmetric pair of Grassmann type with Lie algebra decomposition $\h = \g \oplus \mathfrak{l} \oplus \mathfrak{p}$ and denote $\m = \g \oplus \mathfrak{p}$, in this section we study the affine structure in $\m$ given by the pseudo-Riemannian structure induced by $B$.

\subsection{Affine connection as a non-associative algebra}

Fix the pseudo-Riemannian metrics in $H$ and $L \bs H$ induced from some rescaling of $B$ and recall that an element $X \in \h$ generates two distinct killing fields in $H$ via left and right multiplication of the group
	\[	X^+_{h} = \frac{d}{dt}_{|_{t=0}} h \cdot exp(tX), \qquad X^*_{h} = \frac{d}{dt}_{|_{t=0}} exp(tX) \cdot h, \qquad \forall \ h \in H.	\]
In the homogeneous space $L \bs H$ only one of these killing fields is well defined, namely the one defined by right multiplications $X^+_{L h} = \frac{d}{dt}_{|_{t=0}} L h \cdot exp(tX)$.

\begin{lema}\label{L3.1}
If $\pi : H \rightarrow L \bs H$ denotes the natural projection, the linear map $d \pi_e$ gives the identification $\m \cong T_{Le} ( L \bs H )$ and under this identification 
	\[ 2 \left( \nabla_{X^+} Y^+ \right)_{Le} = [X,Y]_\m, \qquad  \forall \ X,Y \in \m, \]
 where $\nabla$ is the Levi-Civita connection in $L \bs H$ induced by the pseudo-Riemannian submersion and $[X,Y]_\m$ is the orthogonal projection of $[X,Y]$ to $\m$.
\end{lema}

\begin{proof}
Take $X \in \m$ and $\{X_j\}$ a $B$-orthonormal basis of $\mathfrak{l}$, then the vector field defined by $\overline{X} = X^+ - \sum_j \overline{g}( X^+ , X_j^* ) X_j^*$ is orthogonal to the $L$-orbits in $H$ and its projection is precisely $X^+$ in $L \bs H$ (here, $\overline{g}$ is the pseudo-Riemannian metric tensor in $L \bs H$ induced by $B$). Moreover, if $D$ denotes the Levi-Civita connection in $H$ and $Y \in \m$, then 
	\[ D_{\overline{X}} \overline{Y} = D_{X^+} Y^+ + \sum_j \left( \overline{g}( X^+ , X_j^* ) Z_j^1 + \overline{g}( Y^+ , X_j^* ) Z_j^2 + a_j X_j^* \right),  \]
for some smooth vector fields $Z_j^s$ and functions $a_j$. Observe that 
	\[ \langle Z^+ , X_j^* \rangle_{e} = B(Z,X_j) = 0, \qquad \forall \ Z \in \m \]
and $2 \left( D_{X^+} Y^+ \right)_e = [X,Y]$ because the metric is bi-invariant \cite[Corollary 11.10]{O}, so we get
	\[ 2 \left(D_{\overline{X}} \overline{Y} \right)_e = [X,Y]_\m + W\]
for some $W \in \mathfrak{l}$, the result now follows from O'Neill's formula for Levi-Civita connections under pseudo-Riemannian submersions \cite[Lemma 7.45]{O}.
\end{proof}

The bilinear form $[\cdot , \cdot ]_\m$ induces in $\m$ the structure of a non-associative algebra that has as its automorphism group
	\[ Aut(\m) = \{ T \in GL(\m) : T[X, Y]_\m = [TX, TY]_\m, \quad \forall \ X,Y \in \m \}.	\]
$Aut(\m)$ is an algebraic Lie group with Lie algebra
	\[	Der(\m) = \{ T \in GL(\m) : T[X, Y]_\m = [TX, Y]_\m + [X, TY]_\m, \ \forall \ X,Y \in \m \}. \]
Consider the isometry group of $L \bs H$ with the given pseudo-Riemannian structure and denote it simply by $Iso(L \bs H)$, consider also the isotropy subgroup of elements that fix the identity class
	\[ Iso(L \bs H, Le) = \{ \varphi \in Iso(L \bs H) : \varphi(Le) = Le \},	\]
then we denote the isotropy representation as
	\[ \begin{array}{rcl}
		\lambda_e : Iso(L \bs H, Le) & \rightarrow & GL(\m) \\
		\varphi & \mapsto & d\varphi_e.
	\end{array}\]

\begin{cor}\label{C3.1}
If $L \bs H$ is connected, $\lambda_e$ is injective with closed image contained in $Aut(\m)$.
\end{cor}

\begin{proof}
The first part follows from the fact that an isometry of a connected pseudo-Riemannian manifold is completely determined by its value and derivative in a single point, see \cite[Sec. I]{Ko}. The second part follows from Lemma \ref{L3.1} and the fact that $\psi_*(\nabla_U V ) = \nabla_{\psi_*(U)} \psi_*(V)$ for every isometry $\psi \in Iso(L \bs H)$ and $U,V$ local vector fields in $L \bs H$.
\end{proof}

\subsection{Automorphisms of the affine algebra}

If $(H,G \times L)$ is a symmetric pair of Grassmann type, then $G \times L$ is reductive (with non-trivial one-dimensional center if $H$ is of complex-hermitian type and semisimple otherwise). Lets consider $G = G_0 \times Z$ as a direct product with $Z$ the center of the group $G \times L$ and $G_0$ a the simple part of $G$. Observe that the decomposition $\h = \g \oplus \mathfrak{l} \oplus \mathfrak{p}$ is $Ad(G \times L)$-stable and in particular $Ad(G \times L) \subset Aut(\m)$, if we define by
	\[	ad_x(y) = [x,y]_\m, \qquad x \in \h, \ y \in \m,	\]
the adjoint representation in $\m$, we have that $ad_x(y) = [x,y]_\m = [x,y]$ for every $x \in \g \oplus \mathfrak{l}$ and $y \in \mathfrak{\m}$, and thus $ad(\g \oplus \mathfrak{l}) \subset Der(\m)$, however $ad_x$ won't be a derivation of $\m$ for a general $x \in \h$ as the following lemma tells us

\begin{lema}\label{lema:3.2.1}
For every $x \in \mathfrak{p}$, if $ad_x \in Der(\m)$, then $x = 0$.
\end{lema}
\begin{proof}
Recall that $\mathfrak{p} \cong \mathbb{F}^n \otimes (\mathbb{F}^m)^*$ is an irreducible representation of $\g \oplus \mathfrak{l}$, i.e. it has no non-trivial invariant subspaces, and 
	\[	\mathfrak{p}_0 = \{ x \in \mathfrak{p} : ad_x \in Der(\m)	\}	\]
is a $\g \oplus \mathfrak{l}$-submodule of $\mathfrak{p}$, so that either $\mathfrak{p}_0 = \mathfrak{p}$ or $\mathfrak{p}_0 = 0$, thus we only need to find an element $x \in \mathfrak{p}$ such that $ad_x \notin Der(\m)$, here are a few fundamental examples:
\begin{enumerate}
	\item\label{ex:1} If $\h = \mathfrak{so}(4)$ and $\g = \mathfrak{l} = \mathfrak{so}(2)$, consider
	\[ x_s = \left(\begin{array}{cc} 0 & A_s \\ - A_s^t & 0 \end{array}\right), \qquad s = 0,1,2,3,	\]
where
	\[ A_0 = \left(\begin{array}{cc} -1 & 0 \\ 0 & 0 \end{array}\right),	\quad	A_1 = \left(\begin{array}{cc} 0 & -1 \\ 0 & 0 \end{array}\right),	\quad
		A_2 = \left(\begin{array}{cc} 0 & 0 \\ 0 & 1 \end{array}\right),	\quad	A_3 = \left(\begin{array}{cc} 0 & 0 \\ -1 & 0 \end{array}\right),
	\]
	so that $ad_{x_0}[x_1,x_2] = x_3$ and $[ad_{x_0}x_1,x_2] + [x_1,ad_{x_0}x_2] = 0$, thus $ad_{x_0}$ is not a derivation.

	\item\label{ex:2} If $\h = \mathfrak{sp}(2)$ and $\g = \mathfrak{l} = \mathfrak{sp}(1)$, consider
	\[ x_\sigma = \left(\begin{array}{cc} 0 & \sigma \\ - \overline{\sigma} & 0 \end{array}\right), \qquad \sigma \in \mathbb{H}	\]
so that $-[x_1,[x_i,x_j]] = [[x_1,x_i],x_j] = [x_i, [x_1,x_j]] = 2 x_k$, where $i,j,k$ are the linearly independent imaginary elements of $\mathbb{H}$, thus $ad_{x_1}$ is not a derivation.

	\item\label{ex:3} If $\h = \mathfrak{su}(2,1)$, $\g = \mathfrak{u}(1)$ and $\mathfrak{l} = \mathfrak{su}(2)$, consider
	\[	X_z = \left(\begin{array}{ccc} 0 & 0 & z \\ 0 & 0 & 0 \\ - \overline{z} & 0 & 0 \end{array}\right), \quad
		Y_z = \left(\begin{array}{ccc} 0 & 0 & 0 \\ 0 & 0 & z \\ 0 & - \overline{z} & 0 \end{array}\right), \quad z \in \mathbb{C},
	\]
so that $[X_1,[Y_1,Y_i]_\m]_\m = -[[X_1,Y_1]_\m,Y_i]_\m = -[Y_1,[X_1,Y_i]_\m]_\m = X_i$, so $ad_{X_1}$ is not a derivation.
\end{enumerate}
Observe that if $\h$ has a simple complexification $\h_\mathbb{C}$, then for every $x \in \mathfrak{p}$, $ad_x$ can be seen as either a real linear map of $\m$ or a complex linear map of $\m_\mathbb{C}$ and moreover $ad_x \in Der(\m)$ if and only if $ad_x \in Der(\m_\mathbb{C})$, so it is enough to find the desired element either in $\mathfrak{p}$ or in $\mathfrak{p}_\mathbb{C}$. Observe also that we can do an induction into succesive dimensions by considering block diagonal homomorphisms 
		\[	M_{k+l}(\mathbb{F}) \rightarrow M_{n+m}(\mathbb{F}), \qquad
				A \mapsto \left(\begin{array}{ccc} 0_k & 0 & 0 \\ 0 & A & 0 \\ 0 & 0 & 0_l \end{array}\right), 	\]
					where $0_k$ and $0_l$ are zero square matrices of dimensions $(n-k) \times (n-k)$ and $(m-l) \times (m-l)$ respectively, so we only need to find the non-derivation element for the smallest posible dimensions in each type. These two observations and example (\ref{ex:1}) gives us the non-derivation element for all the orthogonal type pairs, i.e. where $\h$ is either $\mathfrak{so}(p,q)$ or $\mathfrak{so}(n+m, \mathbb{C})$. The successive inclusions $\mathbb{R} \subset \mathbb{C} \subset \mathbb{H}$ give us injective homomorphisms
	\[	M_{N}(\mathbb{R}) \hookrightarrow M_{N}(\mathbb{C}) \hookrightarrow M_{N}(\mathbb{H}),	\]
that gives us the non-derivation elements for the hermitian types where $\h$ is $\mathfrak{su}(p,q)$ and $\mathfrak{sp}(p,q)$ except in the smallest dimension that don't restrict to an orthogonal type, these cases are covered by exemples (\ref{ex:2}) and (\ref{ex:3}) (the cases where $\h$ is either $\mathfrak{su}(3)$ or $\mathfrak{sp}(1,1)$ are obtained by passing to the complexification). Finally we observe that $\mathfrak{sp}(n)$ complexifies to $\mathfrak{sp}(2n,\mathbb{C})$, the same as $\mathfrak{sp}(2n,\mathbb{R})$, so the complexification argument gives us again the result in these cases and we get the result for every possibility.
\end{proof}

\begin{lema}\label{lema:3.2.2}
If $T : \m \rightarrow \m$ is a $\g_0$-homomorphism with respect to the adjoint representation, then it preserves the decomposition $\m = \g_0 \oplus \mathfrak{z} \oplus \mathfrak{p}$, where $\mathfrak{z} = Lie(Z)$ is the center of $\g \oplus \mathfrak{l}$.
\end{lema}
\begin{proof}
We have that $[\g_0,\mathfrak{z}] = 0$ and $\g_0$ is irreducible as $\g_0$-module, moreover  
	\[		\mathfrak{p} \cong \mathbb{F}^n \otimes (\mathbb{F}^m)^* = \bigoplus_j V_j,	 \]
where $V_j \cong \mathbb{F}^n$ is irreducible not isomorphic to $\g_0$. If $U,W \subset \h$ are $\g_0$-submodules, denote by
	\[	i_U : U \rightarrow \h, \qquad \pi_W : \h \rightarrow W	\]
the inclusion and the orthogonal projection, so that $T^U_W = \pi_W \circ T \circ i_U$ is again a $\g_0$-homomorphism and thus $Ker(T^U_W)$ and $Im(T^U_W)$ are $\g_0$-submodules of $U$ and $W$ correspondingly. As $\mathfrak{z}$, $V_j$ and $\g_0$ are non-isomorphic $\g_0$-modules without non-trivial $\g_0$-submodules, we will have that $T^Z_W \neq 0$ only when they are isomorphisms and thus only on the cases
	\[	T^{\mathfrak{z}}_{\mathfrak{z}}, \quad T^{V_i}_{V_j} \quad \textrm{and} \quad T^{\g_0}_{\g_0},	\]
this implies that $T(\mathfrak{z}) \subset \mathfrak{z}$, $T(\mathfrak{p}) \subset \mathfrak{p}$ and $T(\g_0) \subset \g_0$ and the result follows.
\end{proof}

\begin{cor}\label{cor:automorphisms}
$Aut(\m)$ has finitely many components and has $Ad(G \times L)$ as a finite index subgroup.
\end{cor}

\begin{proof}
	As an automorphism group of a non-associative algebra over $\mathbb{R}$, $Aut(\m)$ is an algebraic group, so it has only finitely many connected components \cite[Thm 3.6]{P-R}, moreover it is a Lie group with Lie algebra $Der(\m)$ and $Ad(G \times L)$ is a Lie subgroup of $Aut(\m)$ with Lie subalgebra $ad(\g \oplus \mathfrak{l})$, so it is enough to prove that $Der(\m) = ad(\g \oplus \mathfrak{l})$. Let $D = Der(\m)$ and observe that for every $\delta \in D$, and $x \in \g$, the equation $[\delta , ad_x] = ad_{\delta (x)}$ holds. If we decompose $\delta(x) = y_1 + y_2$, where $y_1 \in \g$ and $y_2 \in \mathfrak{p}$, observe that 
	\[	ad_{y_2} = ad_{\delta (x)} - ad_{y_1} \in D	\]
and by Lemma \ref{lema:3.2.1}, $y_2 = 0$, this tells us $[\delta,ad_x] \in ad(\g)$ for every $x \in \g$ and thus $ad(\g)$ is an ideal of $D$. As $\g_0$ is simple, $ad(\g_0)$ is a simple subalgebra of $D$, so the Killing form of $D$ restricted to $ad(\g_0)$ is non-degenerated and thus we have a direct sum decomposition 
	\[	ad(\g_0) \oplus I = D,	\]
	where $I$ is the orthogonal complement of $ad(\g_0)$ in $D$. As the killing form is $ad$-invariant, $I$ is also an ideal of $D$ such that $I \cap ad(\g_0) = 0$, so we have $[\delta, ad_x] = 0$ for every $\delta \in I$ and $x \in \g_0$, thus $I$ is a trivial $\g_0$-module when considered with the $ad$-action and as $\mathfrak{z} \oplus \mathfrak{l}$ is also a trivial $\g_0$-submodule of $\h$, the uniqueness of the $\g_0$-module decomposition of $D$ tells us that $ad(\mathfrak{z} \oplus \mathfrak{l}) \subset I$. Fix an element $\delta \in I$, we have that $[\delta, ad_x] = 0$ for every $x \in \g_0$, so that
		\[	\delta : \m \rightarrow \m	\]
	is a $\g_0$-homomorphism and thus by Lemma \ref{lema:3.2.2}, $\delta$ preserves the decomposition $\m = \g_0 \oplus \mathfrak{z} \oplus \mathfrak{p}$. Observe that for every non-zero element $x \in \g_0 \oplus \mathfrak{z}$, $ad_x$ restricts to a non-trivial linear map of $\mathfrak{p}$, thus the equation $[\delta, ad_x] = 0 = ad_{\delta(x)}$ for every $x \in \g_0 \oplus \mathfrak{z}$ tells us that in fact $\delta(\g_0 \oplus \mathfrak{z}) = 0$ and thus we can consider the derivation as just the restricted linear map
	\[	\delta : \p \rightarrow \p,	\]
that will commute also with the $ad(\mathfrak{z})$-action, that is, $\delta$ is a $\g$-homomorphism. So, we have an inclusion of Lie algebras
	\[	ad(\mathfrak{z} \oplus \mathfrak{l}) \subset I,	\]
	where $\delta \in I$ is a $\g$-homomorphism of $\mathfrak{p}$, and from the fact that $[\mathfrak{p},\mathfrak{p}] \subset \g$, $\delta$ is a derivation of the bilinear form
	\[	\mathcal{B} : \mathfrak{p} \times \mathfrak{p} \rightarrow \g, \]
	induced from the Lie bracket, but as $\delta(\g) = 0$, then $D_\delta \mathcal{B} = 0$ and Proposition \ref{prop:2-1} tells us thus that $I \cong \mathfrak{z} \oplus \mathfrak{l}$, so that in fact $I = ad(\mathfrak{z} \oplus \mathfrak{l})$, $D = ad(\g \oplus \mathfrak{l})$ and the result follows.
\end{proof}

\section{Proof of main theorem}

In this section $(H, G \times L)$ denotes a symmetric pair of Grassmann type with $H$ a connected simple Lie group as in Section \ref{sec:stief}, the case where $H$ is non-connected is covered by looking at its connected components as discused in Remark \ref{rem:non-connected}.

\subsection{The isometry group}
Recall the multiplication homomorphism $M : G \times H \rightarrow Iso(L \bs H)$ given by $M(g,h) = L_g \circ R_{h^{-1}}$ and denote by 
	\[	C : G \times L \subset H \rightarrow Iso(L \bs H, e L), \qquad C(h) = M(h,h),	\]
the conjugation homomorphisms, i.e. $C(h)(L h') = L (hh'h^{-1})$.

\begin{lema}\label{lema:4.1}
$Iso(L \bs H, Le)$ has finitely many connected components and has
		\[	M(G \times H) \cap Iso(L \bs H, eL) = C(G \times L)	\]
as a finite index subgroup.
\end{lema}

\begin{proof}
Observe that for there exists an automorphism $\theta \in Aut(H)$ such that 
	\[	G \times L = \{g \in H : \theta(g) = g \},	\]
such automorphism is of the form $\theta(g) = I_{k,l} g I_{k,l}$ for some $k,l \in \mathbb{N}$ and
	\[	I_{k,l} = \left(\begin{array}{cc} I_k & 0 \\ 0 & -I_l	\end{array}\right).	\]
Take $(g,h) \in G \times H$ such that $M(g,h) \in Iso(L \bs H, eL)$, that is, such that $L_g \circ R_{h^{-1}} (Le) = L e$, then we have that $gh^{-1} \in L$, and thus 
	\[	g h^{-1} = \theta(g h^{-1}) = \theta(g) \theta(h^{-1}) = g \theta(h^{-1}).	\]
This tells us that $\theta(h^{-1}) = h^{-1}$, so $h^{-1} = g_1^{-1} l_1^{-1}$ for $g_1 \in G$ and $l_1 \in L$, as $gg_1^{-1} l_1^{-1} \in L$ and $G \cap L = \{e\}$, we have that $gg_1^{-1} = e$ and thus
	\[	M(g,h) = M(g,gl_1) = C(gl_1), \]
because $C(l_1) = M(e,l_1)$, thus we have that
	\[	M(G \times H) \cap Iso(L \bs H, eL) = C(G \times L).	\]
Recall the isotropy representation $\lambda_e : Iso(L \bs H, Le) \rightarrow Aut(\m)$ and observe that $\lambda_e \circ C(G \times L) = Ad(G \times L)$, so we have an inclusion of Lie groups
	\[	Ad(G \times L) \subset \lambda_e(Iso(L \bs H, Le)) \subset Aut(\m),	\]
thus by Corollary \ref{C3.1} and Corollary \ref{cor:automorphisms}, we have that $Iso(L \bs H, Le)$ has finitely many components sharing the same connected component of the identity as the subgroup $C(G \times L)$ and the result follows.
\end{proof}

\begin{proof}[Proof of Theorem \ref{Thm:main}]
If we consider the action of $G \times H$ via left and right multiplications in $L \bs H$, Lemma \ref{lema:4.1} tells us that the isotropy subgroup is precisely $C(G \times L)$ and thus we have the identification
	\begin{equation}\label{equiv1}	L \bs H \cong M(G \times H) / C(G \times L).	\end{equation}
As the isotropy subgroup $Iso(L \bs H, L e)$ has finitely many components by Lemma \ref{lema:4.1} and $Iso(L \bs H)$ is a Lie group acting on the connected manifold $L \bs H$, then  $Iso(L \bs H)$ also has finitely many components, see for example \cite[Lemma 2.1]{Q1}. The connected component of the identity acts transitively on $L \bs H$ and thus we have
	\begin{equation}\label{equiv2}	L \bs H \cong Iso(L \bs H)_0 / \left( Iso(L \bs H, L e) \cap Iso(L \bs H)_0	\right).	\end{equation}
Observe that passing to a finite index subgroup doesn't change the dimension in Lie groups, so Lemma \ref{lema:4.1} implies that 
	\begin{equation}\label{equiv3}	
		dim(Iso(L \bs H, L e) \cap Iso(L \bs H)_0) = dim(Iso(L \bs H, L e)_0) = dim(C(G \times L) ).
	\end{equation}
We have thus an inclusion of Lie groups $M(G \times H) \subset Iso(L \bs H)$ having finitely many components and by (\ref{equiv1}), (\ref{equiv2}) and (\ref{equiv3}), having the same dimension, thus $M(G \times H)$ is a finite index subgroup of $Iso(L \bs H)$. Finally,
	\[	M: G \times H \rightarrow Iso(L \bs H)	\]
	is a continuous homomorphism that is non-trivial in each factor of $Z(G) \times [G,G] \times H$, thus $M$ is in fact a covering with discrete Kernel contained in the center of $G \times H$. As the center of $[G,G] \times H$ is finite and $Z(G)$ is compact, then $Ker(M)$ is finite and the result follows.
\end{proof}

\subsection{Clifford-Klein forms}

\begin{proof}[Proof of Corollary \ref{Thm:Cliff-Klein}]
Let $\Gamma \subset Iso(L \bs H)$ be a discrete subgroup so that we have a $G$-equivariant, pseudo-Riemannian covering
	\[	L \bs H \rightarrow \left( L \bs H \right) / \Gamma.	\]
Observe that a finite index subgroup of $\Gamma$ fixes the connected component of $L \bs H$ so we may suppose from the start that $L \bs H$ is connected and thus Theorem \ref{Thm:main} tells us that $\Lambda = \Gamma \cap M(G \times H)$ is a finite index subgroup of $\Gamma$. If $\gamma \in \Lambda$, there are elements $g \in G$ and $h \in H$ such that $\gamma = L_g \circ R_h$ and by the $G$-equivariance of the covering, we have that if $g' \in G$, 
	\[	(L_g \circ R_h) \circ L_{g'} = \gamma \circ L_{g'} = L_{g'} \circ \gamma = 
	L_{g'} \circ (L_g \circ R_h),	\]
that is
	\begin{equation}\label{eq:commut}	L_{gg'} \circ R_h = L_{g'g} \circ R_h, \qquad \forall \ g' \in G.	\end{equation}
Consider the subgroup $\Gamma_1 \subset G \times H$ determined by 
	\[	\Gamma_1 = \{ (g,h) \in G \times H : L_g \circ R_h \in \Lambda \},	\]
that by property (\ref{eq:commut}) is contained in $Z(G) \times H$. If $Z(G)$ is finite, it is enough to take
	\[	\Gamma_0 = \Gamma_1 \cap \left( \{e\} \times H \right)	\]
that will be a finite index subgroup of $\Gamma_1$ and thus we have a finite covering 
	\[	\left(L \bs H \right) / \Gamma_0 \rightarrow \left(L \bs H \right) / \Gamma	\]
that preserves the properties of compactness and finite volume. If $Z(G)$ is not finite, we need to do a more involved procedure: Take the projection into the second coordinate
	\[	p : \Gamma_1 \rightarrow H, \qquad p(g,h) = h,	\]
then $\Gamma_0 = p(\Gamma_1)$ is a discrete subgroup of $H$ because $Z(G)$ is compact. The subgroup
	\[	F = Ker(p) = \Gamma_1 \cap \left(Z(G) \times \{e\} \right)	\]
is compact and discrete, thus it is finite and we have the exact sequence of groups
	\[	1 \rightarrow F \rightarrow \Gamma \rightarrow \Gamma_0 \rightarrow 1.	\]
If we write $\Gamma_2 = Z(G) \times \Gamma_0 \supset \Gamma_1$, we have an open surjective map
	\[	\left(L \bs H \right) / \Gamma_1 \rightarrow \left(L \bs H \right) / \Gamma_2, \qquad
		x \Gamma_1 \mapsto x \Gamma_2,	\]
so if $\left(L \bs H \right) / \Gamma_1$ is compact (has finite volume), then $\left(L \bs H \right) / \Gamma_2$ is compact (has finite volume). We also have a fibration with compact fiber
	\[	Z(G) \rightarrow L \backslash H / \Gamma_0 \rightarrow (Z(G) \times L) \backslash H / \Gamma_0 = 
	\left(L \bs H \right) / \Gamma_2,	\]
so if $\left(L \bs H \right) / \Gamma_2$ is compact (has finite volume), then $L \backslash H / \Gamma_0$ is compact (has finite volume). Finally, in the case where $L$ is compact and $L \bs H / \Gamma_0$ has finite volume, then we have a fibration
	\[	H / \Gamma_0 \rightarrow L \bs H / \Gamma_0	\]
with compact fiber $L$, so the finite volume property lifts and this implies that $\Gamma_0$ is a lattice, as stated.
\end{proof}

\end{document}